\documentclass[10pt,a4paper]{amsart}
\usepackage[english]{babel}
\usepackage{amssymb,xypic,amscd,amsthm}
\CompileMatrices
\newtheorem{defn}{Definition}
\newtheorem{thm}[defn]{Theorem}
\newtheorem{cor}[defn]{Corollary}
\newtheorem{lem}[defn]{Lemma}
\newtheorem{prop}[defn]{Proposition}

\theoremstyle{plain}
\newtheorem{rem}[defn]{Remark}
\theoremstyle{remark}
\newtheorem{exam}{Example}
\numberwithin{equation}{section} \numberwithin{defn}{section}

\newcommand\ed{\operatorname{End}}
\renewcommand\Im{\operatorname{Im}}
\newcommand\Ker{\operatorname{Ker}}

\begin{document}

\title[Infinite Moore-Penrose Inverse]{Moore-Penrose Inverse of some Linear Maps \\ on Infinite-Dimensional Vector Spaces}
\author{V\'ictor Cabezas S\'anchez (*)  \\ Fernando Pablos Romo (**)}

\address{ Departamento de Matem\'aticas, Universidad de Salamanca, Plaza de la Merced 1-4, 37008 Salamanca, Espa\~na}
 \email{ (*) v.cabezas@usal.es}
 \email{ (**) fpablos@usal.es}

\thanks{
\\ This work is partially supported by a collaboration-fellowship of the Spanish Government (*), and by the
Spanish Government research projects nos. MTM2015-66760-P and PGC2018-099599-B-I00 and the Regional Government of Castile and Leon research project no. J416/463AC03 (**)}

\begin{abstract} The aim of this work is to characterize linear maps of inner pro\-duct infinite-dimensional vector spaces where the Moore-Penrose inverse exists. This MP inverse generalizes the well-known  Moore-Penrose inverse of a matrix $A\in  \text{Mat}_{n\times m} ({\mathbb C})$. Moreover, a method for the computation of the MP inverse of some endomorphisms on infinite-dimensional vector spaces is given. As an application, we study the least norm solution of an infinite linear system from the Moore-Penrose inverse offered.
\end{abstract}

\maketitle

\bigskip

\setcounter{tocdepth}1

\tableofcontents
\bigskip

\section{Introduction}

 Given again a matrix $A\in  \text{Mat}_{n\times m} ({\mathbb C})$, the Moore-Penrose inverse of $A$ is the unique matrix $A^\dagger \in \text{Mat}_{m\times n} ({\mathbb C})$ such that:

\begin{itemize}

\item $A \, A^\dagger \, A = A$;

\item $A^\dagger \, A \, A^\dagger = A^\dagger$;

\item $(A \, A^\dagger)^* = A \, A^\dagger$;

\item $(A^\dagger \, A)^* = A^\dagger \, A$;

\end{itemize}

$B^*$ being the  conjugate transpose of the matrix $B$.

The Moore-Penrose inverse of $A$ always exists, it is a reflexive generalized inverse of $A$, $[A^\dagger]^\dagger = A$ and, if $A \in \text{Mat}_{n\times n} ({\mathbb C})$ is non-singular, then $A^\dagger$  coincides with the inverse matrix $A^{-1}$. 

Recently, generalized inverses of matrices $A\in  \text{Mat}_{n\times m} ({\mathbb C})$ have been extended to some linear maps on infinite-dimensional vector spaces. Indeed, the authors have computed explicit solutions of infinite systems of linear equations from reflexive generalized inverses of finite potent endomorphisms in \cite{CaP2}  and, also, the second-named author has generalized the notion of Drazin inverse to finite potent endomorphisms in \cite{Pa-Dr}.

The aim of this work is to characterize linear maps of inner product infinite-dimensional vector spaces where the Moore-Penrose (MP) inverse exists. This MP inverse generalizes the Moore-Penrose inverse $A^\dagger$ of a matrix $A\in  \text{Mat}_{n\times m} ({\mathbb C})$. Moreover,  a method for the computation of the MP inverse of some endomorphisms on infinite-dimensional vector spaces is given. As an application, we study the least norm solution of an infinite linear system from the Moore-Penrose inverse offered.

The paper is organized as follows.  In section \ref{s:pre} we recall the basic definitions of this work:  inner product vector spaces, finite potent endomorphisms, reflexive generalized inverse and Moore-Penrose inverse of a  $(n\times m)$-matrix. Also, in this section, we  briefly describe the construction of Jordan bases for endomorphisms admitting an annihilator polynomial.

Section \ref{s:mo-l} contains the main results the this work: the definition of linear map admissible for the Moore-Penrose inverse (Definition \ref{def:inf-Mo-Pen-adm}), the proof of the existence and uniqueness of the MP inverse for these linear maps (Theorem \ref{th:MThTN}) and the conditions for computing the MP inverse for some endomorphisms on infinite-dimensional vector spaces from the MP inverses of $(n\times n)$-matrices (Theorem \ref{t:char}).

Finally, Section \ref{s:study} is devoted to study infinite systems of linear equations from the Moore-Penrose Inverse. Thus, Proposition \ref{p:study-least} shows that if $(V,g)$ and $(W,{\bar g})$ are two arbitrary inner product vector spaces over $\mathbb R$ of $\mathbb C$, $f \colon V \to W$ is a linear map admissible for the Moore-Penrose inverse and $f (x) = w$ is a linear system, then $f^\dagger (w)$ is the unique minimal least ${\bar g}$-norm solution of this linear system.

\medskip

\section{Preliminaries} \label{s:pre}

This section is added for the sake of completeness.

\subsection{Inner Product Vector Spaces} \label{ss:inner-product}

Let $k$ be the field of the real numbers or the field of the complex numbers, and let $V$ be a $k$-vector space.

    An inner product on $V$ is a map $g \colon V \times V \to k$ satisfying that:

\begin{itemize}

\item $g$ is linear in its first argument: $$g(\lambda v_1 + \mu v_2, v') = \lambda g(v_1,v') + \mu g(v_2,v')  \text{ for every }  v_1,v_2, v' \in V\, ;$$

\item $g (v', v) = {\overline {g (v,v')}}$ for all $v,v' \in V$, where ${\overline {g (v,v')}}$ is the complex conjugate of $g(v,v')$;

\item  $g$ is positive definite: $$g(v,v) \geq 0 \text{ and } g(v,v) = 0 \Longleftrightarrow v = 0\, .$$

\end{itemize}

 Note that $g(v,v) \in \mathbb R$  for each $v\in V$, because $g(v,v) = {\overline {g(v,v)}}$.

A pair $(V,g)$ is named ``inner product vector space''.

If $(V, g)$ is an inner product vector space over $\mathbb C$, it is clear that $g$ is antilinear in its second argument, that is: $$g(v, \lambda v'_1 + \mu v'_2) = {\bar \lambda} g(v,v'_1) + {\bar \mu} g(v,v'_2)$$\noindent for all $v, v'_1,v'_2 \in V$, and $\bar \lambda$ and $\bar \mu$ being the conjugates of $\lambda$ and $\mu$ respectively.

Nevertheless, if $(V$, g) is an inner product vector space over $\mathbb R$, then $g$ is symmetric and bilinear.

The norm on an inner product vector space $(V,g)$ is the real-valued function $$\begin{aligned} \Vert \cdot \Vert_g \colon V &\longrightarrow \mathbb R \\ v &\longmapsto + \sqrt{g(v,v)} \, , \end{aligned}$$\noindent and the distance is the map $$\begin{aligned} d_g \colon V \times V &\longrightarrow \mathbb R \\ (v,v') &\longmapsto \Vert v' - v \Vert_g \, . \end{aligned}$$

Simple examples of  inner product vector spaces are Euclidean  finite-dimensional real vector spaces and complex Hilbert spaces.

Let us now consider two inner product vector spaces: $(V,g)$ and $(W,{\bar g})$. If $f\colon V \to W$ is a linear map, a linear operator $f^*\colon W \to V$ is called the adjoint of $f$ when $$g (f^*(w), v) = {\bar g} (w, f(v))\, ,$$\noindent for all $v\in V$ and $w\in W$. If $f\in \ed_k (V)$, we say that $f$ is self-adjoint when $f^* = f$.

Moreover, if  $(V,g)$ and $(W,{\bar g})$ are finite-dimensional  inner vector spaces over $\mathbb C$, $B = \{v_1, \dots, v_m\}$ and $B' = \{w_1, \dots, w_n\}$ are orthonormal bases of $V$ and $W$ respectively, $f\colon V \to W$ is a linear map, $A\in \text{Mat}_{n\times m} ({\mathbb C})$ and $f\equiv A$ in these bases, then $f^* \equiv A^*\in \text{Mat}_{m\times n} ({\mathbb C})$ in the same bases, where $A^*$ is the conjugate transpose of $A$.

\medskip

\subsection{Jordan bases for endomorphisms admitting an annihilator polynomial} \label{s:jordan-basis-anh}

\medskip

Let $V$ be an arbitrary vector space over a ground field $k$, and let $f\in \ed_k(V)$ be an endomorphism of $V$ admitting an annihilator polynomial $$a_f(x) = p_1(x)^{n_1}\cdot \dots \cdot p_r(x)^{n_r}\, ,$$\noindent $p_i(x)$ being irreducible polynomials in $k[x]$.

For each $j\in \{1,\dots,r\}$, we can  consider
$$\nu_i (V,p_j(f)) = \dim_{K_j} \big ( \Ker p_j(f)^i \big / [\Ker p_j(f)^{i-1} + p_j(f) \Ker p(f)^{i+1}] \big )\, ,$$\noindent with $K_j = k[x] \big / p_j(x)$. Henceforth, $S_{{\nu}_i (V,p_j(f))}$ will be a set such that $\# S_{{\nu}_i (V,p_j(f))} = {\nu}_i (V,p_j(f))$, with $S_{{\nu}_i (V,p_j(f))} \cap S_{{\nu}_h (V,p_j(f))} = \emptyset$ for $i\ne h$.

According to the statements of \cite{Pa7} there exist families of vectors $\{v^{ij}_h\}_{ h \in S_{{\nu}_i (V,p_j(f))}}$ with
 $$\begin{aligned} v^{ij}_h &\in \Ker p_j(f)^i \\ v^{ij}_h &\notin \Ker p_j(f)^{i-1} + p_j(f) \Ker p_j(f)^{i+1}\, ,\end{aligned}$$\noindent for all $1\leq j \leq r$ and $1\leq i \leq n_j$, such that if we set
 $$H^{ij}_h = <v^{ij}_h>_f = {\underset {0\leq s \leq i - 1} \bigcup}  \{p_j(f)^{s}[v^{ij}_h], p_j(f)^{s} [f(v^{ij}_h)], \dots , p_j(f)^{s}[f^{d_j -1}(v^{ij}_h)]\}\, ,$$\noindent then
$$\underset {\begin{aligned}1 &\leq j \leq r  \\ 1 &\leq i \leq n_j \\  h &\in S_{{\nu}_i (V,p_j(f))} \end{aligned}} {\bigcup} <v^{ij}_h>_f$$\noindent is a Jordan basis of $V$ for $f$, and this basis determines a decomposition \begin{equation} \label{eq:jordan-admi-poly} V = \underset {\begin{aligned}1 &\leq j \leq r  \\ 1 &\leq i \leq n_j \\  h &\in S_{{\nu}_i (V,p_j(f))} \end{aligned}} {\bigoplus} H^{ij}_h\,  .\end{equation}

\medskip

\begin{exam}[Jordan bases for a nilpotent endomorphism] \label{ex:nilpotent-basis}

Let $V$  again be a vector space over an arbitrary field $k$ and let $f\in\ed_k (V)$ be a nilpotent endomorphism. If n is the nilpotency index of $f$,  setting $W_i^f = \Ker f^i/[\Ker f^{i-1} + f(\ker f^{i+1})]$ with $i\in \{1,2,\dots, n\}$, $\mu_i (V,f) = \text{dim}_k W_i^f$ and $S_{\mu_i (V,f)}$ a set such that $\# S_{\mu_i (V,f)} = \mu_i (V,f)$ with $S_{\mu_i (V,f)} \cap S_{\mu_j (V,f)} = \emptyset$ for all $i \ne j$, one has that there exists a family of vectors $\{{ {v_{s_i}}}\}$  that determines a Jordan basis of $V$ for $f$: \begin{equation} \label{eq:jordan-basis-B} B = \underset {\begin{aligned} s_i &\in S_{{\mu}_i (V,f)} \\ 1 &\leq i \leq n \end{aligned}} {\bigcup} \{{ {v_{s_i}}}, f ({{ v_{s_i}}}), \dots , f^{i-1} ({ {v_{s_i}}})\}\, .\end{equation}
Moreover, if we write $H_{s_i}^f = \langle { {v_{s_i}}}, f ({{ v_{s_i}}}), \dots , f^{i-1} ({ {v_{s_i}}}) \rangle$, the basis $B$ induces a decomposition \begin{equation} \label{eq:decomp} V =  \underset {\begin{aligned} s_i &\in S_{{\mu}_i (V,f)} \\ 1 &\leq i \leq n \end{aligned}} \bigoplus H_{s_i}^f\, .\end{equation}

\end{exam}

\medskip

\subsection{Finite Potent Endomorphisms} \label{ss:finite-potent}

Let $k$ be an arbitrary field, let $V$ be a $k$-vector space and let $\varphi \in \ed_k (V)$. We say
that $\varphi$ is ``finite potent'' if $\varphi^n V$ is finite
dimensional for some $n$. This definition was introduced by J. Tate in \cite{Ta} as a basic tool for his elegant definition
of Abstract Residues.

 In 2007 M. Argerami, F. Szechtman and R. Tifenbach showed in \cite{AST} that an endomorphism $\varphi$ is
finite potent if and only if $V$ admits a $\varphi$-invariant
decomposition $V = U_\varphi \oplus W_\varphi$ such that
$\varphi_{\vert_{U_\varphi}}$ is nilpotent, $W_\varphi$ is finite
dimensional, and $\varphi_{\vert_{W_\varphi}} \colon W_\varphi
\overset \sim \longrightarrow W_\varphi$ is an isomorphism.

\medskip

\subsection{Reflexive Generalized Inverses} \label{ss:reflexive}

Let $\mathbb C$ be the field of complex numbers. Given a matrix $A\in \text{Mat}_{n\times m} ({\mathbb C})$, a reflexive generalized inverse of $A$ is a matrix $A^+ \in   \text{Mat}_{m\times n} ({\mathbb C})$ such that:

\begin{itemize}

\item $A \, A^+ \, A = A$;

\item $A^+ \, A \, A^+ = A^+$.

\end{itemize}

In general, the reflexive generalized inverse of a matrix $A$ is not unique.

\medskip

The notion of reflexive generalized inverse in arbitrary vector spaces is the following:

\begin{defn} \label{df:gen-lin-map} If $V$ and $W$ are $k$-vector spaces, given a morphism $f\colon V \to W$, a linear map $f^+ \colon W \to V$ is a ``reflexive generalized inverse'' of $f$ when:

\begin{itemize}

\item $f \circ f^+ \circ f = f$;

\item $f^+ \circ f \circ f^+ = f^+$.

\end{itemize}

\end{defn}

For every reflexive generalized inverse $f^+$ of $f$, if $w\in W$, it is known that \begin{equation} \label{eq:char-im-rel} w\in \text{Im } f \Longleftrightarrow (f\circ f^+) (w) = w\, .\end{equation}

\subsection{Moore-Penrose Inverse of an $(n\times m)$-Matrix} \label{ss:Moore-finite-prem}

 Given again a matrix $A\in  \text{Mat}_{n\times m} ({\mathbb C})$, the Moore-Penrose inverse of $A$ is a matrix $A^\dagger \in \text{Mat}_{m\times n} ({\mathbb C})$ such that:

\begin{itemize}

\item $A \, A^\dagger \, A = A$;

\item $A^\dagger \, A \, A^\dagger = A^\dagger$;

\item $(A \, A^\dagger)^* = A \, A^\dagger$;

\item $(A^\dagger \, A)^* = A^\dagger \, A$;

\end{itemize}

$B^*$ being the  conjugate transpose of the matrix $B$.

The Moore-Penrose inverse of $A$ always exists, it is unique, it is a reflexive generalized inverse of $A$, $[A^\dagger]^\dagger = A$ and, if $A \in \text{Mat}_{n\times n} ({\mathbb C})$ is non-singular, then the Moore-Penrose inverse of $A$  coincides with the inverse matrix $A^{-1}$. 

For details, readers are referred to \cite{CM}.

\medskip

\section{Moore-Penrose Inverse of Linear Maps on Infinite-Dimensional Vector Spaces} \label{s:mo-l}

\medskip

Henceforth, given a family of subspaces $\{H_i\}_{i\in I}$ of an arbitrary vector space $V$ over a field $k$,  we shall write $V = \underset {i\in I} \oplus H_i$ to indicate that the natural morphism $$\begin{aligned} \underset {i\in I} \oplus H_i &\longrightarrow V \\ (v_i) &\longmapsto \sum_{i\in I} v_i\, ,\end{aligned}$$\noindent is an isomorphism.

\medskip

This section is devoted to proving the existence of a Moore-Penrose inverse of some linear maps on infinite-dimensional vector spaces, such that it generalizes the notion and the properties of the Moore-Penrose inverse of an $(n\times m)$-matrix with entries in  $\mathbb C$. Our generalization will be valid for linear maps on inner product vector spaces over $k = \mathbb R$ and $k = \mathbb C$ and we shall give conditions for computing the Moore-Penrose inverse of endomorphisms on infinite-dimensional vector spaces.

To do so, we shall first study some properties of the Moore-Penrose inverse  of an endomorphism on a finite-dimensional inner product space.

\medskip

\subsection{Moore-Penrose inverse of an endomorphism on finite-dimensional inner product vector spaces}

Let us now consider a finite dimensional inner product vector space $(E,g)$ over $k = \mathbb R$ or $k = \mathbb C$.

If $f  \in \ed_{k} (E)$, one has that $E = \Ker f \oplus [\Ker f]^\perp = \Im f \oplus [\Im f]^\perp$, and there exists an isomorphism $$f \colon [\Ker f]^\perp \overset {\sim} \longrightarrow \Im f\, .$$

Thus, the Moore-Penrose of $f$ is the unique  linear map $f^\dagger \in \ed_k (E)$ such that \begin{equation} \label{def:Moore-Penrose}  f^\dagger (e) = \left \{ \begin{aligned} (f_{\vert_{ [\Ker f]^\perp}})^{-1} (e) \quad &\text{ if } \quad e\in \Im f \\ 0 \quad \quad &\text{ if } \quad e\in [\Im f]^\perp \end{aligned} \right . \, .\end{equation}

\medskip

It is known that $f^\dagger$ is the unique linear map such that:

\begin{itemize}

\item $f^\dagger$ is a reflexive generalized inverse of $f$;

\item $f^\dagger \circ f$ and $f \circ f^\dagger$ are self-adjoint.

\end{itemize}

\smallskip

\begin{defn} \label{def:descomp123} If  $f  \in \ed_{k} (E)$ and ${\mathcal H}_f = \{H_1, \dots, H_n\}$ is a family of subspaces of $E$ invariants for $f$ such that $E = H_1 \oplus \dots \oplus H_n$, we define the endomorphism $f^+_{{\mathcal H}_f} \in \ed_{k} (E)$ as the unique linear map such that $$[f^+_{{\mathcal H}_f}]_{\vert_{H_i}} = {[f_{\vert_{H_i}}]^\dagger} \quad \text{ for each } \quad i\in \{1, \dots, n\}\, .$$

\end{defn}

\medskip

If we denote $f_i = f_{\vert_{H_i}}$ for each $i\in \{1, \dots, n\}$, it is clear that for every vector $e\in E$, such that $e = h_{j_1} + \dots + h_{j_n}$ with $h_{j_i} \in H_{i}$, then $$f_{{\mathcal H}_f}^+ (e) =  f_{1}^\dagger (h_{j_1}) + \dots +  f_{n}^\dagger (h_{j_n})\, .$$

\smallskip

Moreover, it  is immediately observed, from Definition \ref{def:descomp123} and from the properties of the Moore-Penrose inverse, that  $f^+_{{\mathcal H}_f}$ is a reflexive generalized inverse of $f$ for every family ${\mathcal H}_f$.

\medskip

Keeping the previous notation and given a subspace $W\subseteq E$, such that $W \subset H_{i}$ for a certain $i\in \{1, \dots, n\}$, we shall denote $$W_i^\perp = \{v_i\in H_i \text{ such that } g(w,v_i) = 0 \text{ for all } w \in W\}\, .$$

\medskip

\begin{lem} \label{l:or-dcomp} If $U\subset E$ is a subspace and $\{U_1, \dots, U_n\}$ are subspaces of $U$ such that $U = U_1 \oplus \dots \oplus U_n$ with $U_i \subseteq H_i$, then $$[U_1]_1^\perp \oplus \dots \oplus [U_n]_n^\perp \subseteq U^\perp \Longleftrightarrow 
[U_i]_i^\perp \subset [\sum_{j\ne i} U_j]^\perp \text{ for all } i\in \{1, \dots, n\}\, .$$

\end{lem}

\begin{proof} 

$$\begin{aligned} \text{If } [U_1]_1^\perp \oplus \dots \oplus [U_n]_n^\perp \subseteq U^\perp &\Longrightarrow   [U_i]_i^\perp \subseteq U^\perp \text{ for all } i\in \{1, \dots, n\} \Longrightarrow \\ &\Longrightarrow  [U_i]_i^\perp \subseteq [\sum_{j\ne i} U_j]^\perp \text{ for all } i\in \{1, \dots, n\}\, . \end{aligned} $$

Conversely, 
$$\begin{aligned}& \text{if } [U_i]_i^\perp \subseteq [\sum_{j\ne i} U_j]^\perp \text{ for all } i\in \{1, \dots, n\} \Longrightarrow \\ &g(v_1 + \dots + v_n, u_1 + \dots + u_n) = 0 \text{ with } v_i \in [U_i]_i^\perp \text{ and } u_i \in U_i \text{ for all } i\in \{1, \dots, n\}  \\ &\Longrightarrow [U_1]_1^\perp \oplus \dots \oplus [U_n]_n^\perp \subseteq U^\perp\, . \end{aligned}$$

\end{proof}

\medskip

\begin{lem} \label{lem:prop-dar-finite} Using the previous notation, we have that:

\begin{enumerate}

\item $\Im f = \Im f_1 \oplus \dots \oplus \Im f_n$;

\item If ${\mathcal H}_f^\perp = [\Im f_1]_1^\perp \oplus \dots \oplus [\Im f_n]_n^\perp$, then $E = \Im f \oplus {\mathcal H}_f^\perp$;

\item $\Ker f = \Ker f_1 \oplus \dots \oplus \Ker f_n$;

\item If ${\widetilde {\mathcal H}_f}^\perp = [\Ker f_1]_1^\perp \oplus \dots \oplus [\Ker f_n]_n^\perp$, then $$E = \Ker f \oplus {\widetilde {\mathcal H}_f}^\perp\, ;$$

\item $f$ induces an isomorphism between ${\widetilde {\mathcal H}_f}^\perp$ and $\Im f$ and  \begin{equation} \label{def:fini-reflx-345} [f^+_{{\mathcal H}_f}] (e) = \left \{ \begin{aligned} (f_{\vert_{\widetilde {\mathcal H}_f^\perp}})^{-1} (e) \quad &\text{ if } \quad e\in \Im f \\ 0 \quad \quad &\text{ if } \quad e\in {\mathcal H}_f^\perp \end{aligned} \right . \, .\end{equation}

\end{enumerate}

\end{lem}

\begin{proof} 

\quad

\begin{enumerate}

\item It is clear that  $\Im f_1 \oplus \dots \oplus \Im f_n \subseteq \Im f$. Moreover, given $e \in \Im f$, if $e = f(e')$ with $e' = e'_{j_1} + \dots + e'_{j_n}$ and $e'_{j_i} \in H_{i_j}$ for all $i\in \{1, \dots, n\}$, then $$e = f(e'_{j_1}) + \dots + f(e'_{j_n}) \in \Im f_1 \oplus \dots \oplus \Im f_n \, ,$$\noindent because $f(v'_{j_i}) \in H_{i}$ for every $i\in \{1, \dots, n\}$.

\item  Since $H_i = \Im f_i \oplus [\Im f_i]_i^\perp$ for all $i\in \{1, \dots, n\}$, we have that $$\begin{aligned} E &=  H_1 \oplus \dots \oplus H_n =   (\Im f_1 \oplus [\Im f_1]_1^\perp) \oplus \dots \oplus  (\Im f_n \oplus [\Im f_n]_n^\perp) = \\ &=  (\Im f_1 \oplus \dots \oplus \Im f_n) \oplus ([\Im f_1]_1^\perp \oplus \dots \oplus [\Im f_n]_n^\perp) = \Im f \oplus {\mathcal H}_f^\perp\, .\end{aligned}$$

\item It is immediate that  $\Ker f_1 \oplus \dots \oplus \Ker f_n \subseteq \Ker f$. Furthermore, if ${\bar e} \in \Ker f$ and $${\bar e} = {\bar e}_{j_1} + \dots + {\bar e}_{j_n}\quad \text{ with } \quad {\bar e}_{j_i} \in H_{i} \quad \text{ for every } i\in \{1, \dots, n\}\, ,$$\noindent since $$ f({\bar e}_{j_1}) + \dots + f({\bar e}_{j_k}) = 0 \text{ and } H_{i} \cap [\sum_{r\ne i} H_{r}] = {0}\, ,$$\noindent we conclude that $f({\bar e}_{j_i}) = 0$ for all $i\in \{1, \dots, n\}$, from which it is deduced that $\Ker f \subseteq\Ker f_1 \oplus \dots \oplus \Ker f_n$.

\item Similar to (2).

\item It is a direct consequence of (2), (4) and the definition of $f^+_{{\mathcal H}_f}$.

\end{enumerate}

\end{proof}

\medskip

\begin{lem} \label{lem:counter-ex} With the previous notation, in general  ${\widetilde {\mathcal H}_f}^\perp \ne [\Ker f]^\perp$ and 
${\mathcal H}_f^\perp \ne [\Im f]^\perp$.

\end{lem}

\begin{proof} The statement is deduced from the following counter-example.

Let $E = \langle v_1, v_2, v_3, v_4 \rangle$ be a inner product $k$-vector space, $B = \{v_1, v_2, v_3, v_4\}$ being an orthonormal basis, and let us consider the endomorphism $f\colon E \to E$ defined as:

$$f(v_i) = \left \{ \begin{aligned} \quad v_1 \qquad \quad &\text{ if } \quad i = 1,2\, ; \\ \quad -2v_1 \qquad \quad &\text{ if } \quad i = 3\, ; \\ v_1 + v_2 + v_3 \quad &\text{ if } \quad i = 4\, .\end{aligned} \right .$$

If we set $H_1 = \langle v_1, v_2 \rangle$ and $H_2 = \langle v_1 + v_2 + v_3, v_4 \rangle$, it is clear that $H_1$ and $H_2$ are $f$-invariant and, if we again denote $f_1 =  f_{\vert_{H_1}}$ and $f_2 =  f_{\vert_{H_2}}$, we have that:

\begin{itemize}

\item $\Im f_1 = \langle v_1 \rangle$ and $[\Im f_1]_1^\perp = \langle v_2 \rangle$;

\item $\Im f_2 = \langle v_1 + v_2 + v_3 \rangle$ and $[\Im f_2]_2^\perp = \langle v_4 \rangle$;

\item $\Im f = \langle v_1, v_1 + v_2 + v_3 \rangle$ and $[\Im f]^\perp = \langle v_2 - v_3, v_4 \rangle$;

\item $\Ker f_1 = \langle v_1 - v_2 \rangle$ and $[\Ker f_1]_1^\perp = \langle v_1 + v_2 \rangle$;

\item $\Ker f_2 = \langle v_1 + v_2 + v_3 \rangle$ and $[\Ker f_2]_2^\perp = \langle v_4 \rangle$;

\item $\Ker f = \langle v_1 - v_2, v_1 + v_2 + v_3\rangle$ and $[\Ker f]^\perp = \langle v_1 + v_2 -  2v_3, v_4 \rangle$.

Accordingly,  ${\widetilde {\mathcal H}_f}^\perp \ne [\Ker f]^\perp$ and 
${{\mathcal H}_f}^\perp \ne [\Im f]^\perp$ in this case.

\end{itemize}

\end{proof}

\medskip

A direct consequence of Lemma \ref{lem:counter-ex} is:

\begin{cor} Given an endomorphism $f  \in \ed_{k} (E)$, in general  $f^+_{{\mathcal H}_f}$ is not the Moore-Penrose inverse of $f$.

\end{cor}

\medskip

\begin{exam} If $f$ and ${\mathcal H}_f = \{H_1, H_2\}$ are as in the proof of Lemma \ref{lem:counter-ex}, a computation shows that 

$$f^+_{{\mathcal H}_f} (v_i) = \left \{ \begin{aligned} \quad \frac12 v_1 +  \frac12 v_2 \qquad \quad &\text{ if } \quad i = 1 \\ \quad v_4 -  \frac12 v_1 -  \frac12  v_2 \qquad \quad &\text{ if } \quad i = 3 \\ \quad 0 \qquad \quad \quad &\text{ if } \quad i = 2,4\end{aligned} \right .$$\noindent and $$f^\dagger (v_i) = \left \{ \begin{aligned} \quad \frac16 v_1 + \frac16 v_2 - \frac13 v_3 \quad \qquad \qquad \quad &\text{ if } \quad i = 1 \\ \quad  -\frac1{12} v_1 - \frac1{12} v_2 + \frac16 v_3 + \frac12 v_4 \qquad \quad &\text{ if } \quad i = 2 \\ \quad  -\frac1{12} v_1 - \frac1{12} v_2 + \frac16 v_3 + \frac12 v_4 \qquad \quad &\text{ if } \quad i = 3 \\ \quad 0  \qquad \qquad  \quad \qquad \quad \quad &\text{ if } \quad i = 4\end{aligned} \right . \, .$$

Readers can easily check that $f^+_{{\mathcal H}_f}$ is a reflexive generalized inverse of $f$ (Definition \ref{df:gen-lin-map}), although it is clear that $f^+_{{\mathcal H}_f} \ne f^\dagger$.

\end{exam}

\medskip

\begin{lem} \label{lem:equ-imf-345} With the above notation, we have that ${\mathcal H}_f^\perp = [\Im f]^\perp$ if and only if $$[\Im f_i]_i^\perp \subseteq [\sum_{j\ne i} \Im f_j]^\perp \text{ for every }i \in \{1, \dots, n\}\, .$$

\end{lem}

\begin{proof} Considering that $E = \Im f \oplus {\mathcal H}_f^\perp$ and  $[\Im f]^\perp \cap \Im f = \{0\}$, the statement is deduced bearing in mind that this condition is equivalent to ${\mathcal H}_f^\perp \subseteq [\Im f]^\perp$ (Lemma \ref{l:or-dcomp}).

\end{proof}

Similarly, one can prove that 

\begin{lem} \label{lem:equ-kerf-345} We have that ${\widetilde {\mathcal H}_f}^\perp = [\Ker f]^\perp$ if and only if $$[\Ker f_i]_i^\perp \subseteq [\sum_{j\ne i} \Ker f_j]^\perp \text{ for every }i \in \{1, \dots, n\}\, .$$

\end{lem}

\medskip

\begin{prop}  If  $f  \in \ed_{k} (E)$ and ${\mathcal H}_f = \{H_1, \dots, H_n\}$ is a family of subspaces of $E$ invariants for $f$ such that $E = H_1 \oplus \dots \oplus H_n$ and $H_i \subseteq [\sum_{j\ne i} H_j]^\perp$ for all $i \in \{1, \dots, n\}$, then $f^+_{{\mathcal H}_f} = f^\dagger$.
\end{prop}

\begin{proof} If $H_i \subseteq [\sum_{j\ne i} H_j]^\perp$ for all $i \in \{1, \dots, n\}$, then the conditions of Lemma \ref{lem:equ-imf-345} and Lemma \ref{lem:equ-kerf-345} hold. Hence, the claim is deduced.

\end{proof}

\medskip

\subsection{Moore-Penrose inverse of linear maps on arbitrary inner product Spaces}

\smallskip

We shall now generalize the notion of Moore-Penrose inverse to some endomorphisms  of arbitrary vector spaces, in particular some infinite-dimensional vector spaces.

Henceforth $(V,g)$ and $(W,{\bar g})$  will be inner product vector spaces over $k$, with $k = \mathbb C$ or $k = \mathbb R$.

\begin{defn} \label{def:inf-Mo-Pen-adm} Given a linear map $f \colon V \to W$, we say that $f$ is admissible for the Moore-Penrose inverse when  $V = \Ker f \oplus [\Ker f]^\perp$ and $W = \Im f \oplus [\Im f]^\perp$. 

\end{defn}

\begin{rem} It is known that there exist infinite-dimensional vector spaces $V$ and vector subspaces $U\subset V$ such that $V \ne U \oplus U^\perp$. In this case, if $V = U\oplus W$, it is clear that the linear map $f_{_U} \in \ed_k (V)$ defined as $$f_{_U} (v) = \left \{\begin{aligned} 0 \quad &\text{ if } \quad v \in U \\ v \quad &\text{ if } \quad v \in W \end{aligned} \right .$$\noindent is not admissible for the Moore-Penrose inverse.
\end{rem}

\begin{thm}[Existence and uniqueness of Moore-Penrose inverse] \label{th:MThTN} If  $(V,g)$ and $(W,{\bar g})$  are inner product spaces over $k$, then $f \colon V \to W$ is a linear map admissible for the Moore-Penrose inverse if and only if there exists a unique linear map $f^\dagger \colon W \to V$ such that:

\begin{enumerate}

\item $f^\dagger$ is a reflexive generalized inverse of $f$;

\item $f^\dagger \circ f$ and $f \circ f^\dagger$ are self-adjoint, that is:

\begin{itemize}

\item $g ([f^\dagger \circ f] (v), v') = g (v, [f^\dagger \circ f] (v')$;

\item ${\bar g} ([f \circ f^\dagger] (w), w') ={\bar  g} (w, [f \circ f^\dagger] (w')$;

\end{itemize}

\end{enumerate}

for all $v, v' \in V$ and $w,w' \in W$. The operator $f^\dagger$ is named the Moore-Penrose inverse of $f$.

\end{thm}

\begin{proof} If $f$ is admissible of the Moore-Penrose inverse (Definition \ref{def:inf-Mo-Pen-adm}), then the restriction $f_{\vert_{ [\Ker f]^\perp}}$ is an isomorphism between $ [\Ker f]^\perp$ and $\Im f$ and there exists an unique linear map satisfying that  
$$f^\dagger (w) = \left \{ \begin{aligned} (f_{\vert_{ [\Ker f]^\perp}})^{-1} (w) \quad &\text{ if } \quad w\in \Im f \\ 0 \quad \quad &\text{ if } \quad w\in [\Im f]^\perp \end{aligned} \right . \, .$$

We shall now check that $f^\dagger$ satisfies the conditions of the statement.

Firstly, since $$(f \circ f^\dagger) (w) = \left \{ \begin{aligned} w \quad &\text{ if } \quad w\in \Im f \\ 0 \quad  &\text{ if } \quad w\in [\Im f]^\perp \end{aligned} \right .$$\noindent and $(f^\dagger \circ f) (v) = v_1$ with $v = v_1 + v_2$ ($v_1 \in [\Ker f]^\perp$ and $v_2 \in \Ker f$), then it is clear that  $f^\dagger$ is a reflexive generalized inverse of $f$ because:

\begin{itemize}

\item $(f \circ f^\dagger \circ f) (v) = f(v)$;

\item $(f^\dagger \circ f \circ f^\dagger) (w) = f^\dagger (w)$.

\end{itemize}

Moreover,

$${\bar g}([f \circ f^\dagger] (w),w') = {\bar  g} (w, [f \circ f^\dagger] (w')) =  \left \{ \begin{aligned} {\bar g} (w,w') \quad &\text{ if } \quad w, w' \in \Im f \\ 0 \quad \quad  &\text{ if } \quad w\in [\Im f]^\perp \\ 0 \quad \quad  &\text{ if } \quad w'\in [\Im f]^\perp   \end{aligned} \right .$$

And,  if $v, v' \in V$ with $v = v_1 + v_2$,  $v' = v'_1 + v'_2$, $v_1, v'_1 \in [\Ker f]^\perp$ and $v_2, v'_2 \in \Ker f$, one has that $$g ([f^\dagger \circ f] (v), v') = g(v_1,v'_1) =  g (v, [f^\dagger \circ f] (v')\, .$$\noindent Hence, we conclude that $f^\dagger$ satisfies the conditions of the Theorem.

For proving the uniqueness of the Moore-Penrose inverse of $f$, let us consider a linear map ${\widetilde f}  \colon W \to V$ such that

\begin{enumerate}

\item ${\widetilde f}$ is a reflexive generalized inverse of $f$;

\item $g ([{\widetilde f} \circ f] (v), v') = g (v, [{\widetilde f} \circ f] (v')$;

\item ${\bar g} ([f \circ {\widetilde f}] (w), w') ={\bar  g} (w, [f \circ {\widetilde f}] (w')$;

\end{enumerate}

for all $v, v' \in V$ and $w,w' \in W$.

A direct consequence of (1)  is that $({\widetilde f} \circ f)^2 = {\widetilde f} \circ f$. Hence ${\widetilde f} \circ f$ is a projection and, since $$\text{Im } f = \text{Im } (f \circ {\widetilde f} \circ f) \subseteq \text{Im }  (f \circ {\widetilde f}) \subseteq \text{Im } f\, ,$$\noindent then $\text{Im }  (f \circ {\widetilde f}) = \text{Im } f$.

Accordingly, given $w \in Im f$, there exists $\bar w \in W$ such that $(f \circ {\widetilde f}) ({\bar w})$, and then $$(f \circ {\widetilde f}) (w) = (f \circ {\widetilde f})^2 (\bar w) = (f \circ {\widetilde f}) (\bar w) = w\, .$$

Furthermore, if $w' \in [\text{Im } f]^\perp$, we have that $$0 = {\bar g} ([f \circ {\widetilde f}]^2 (w'), w') = {\bar g} ([f \circ {\widetilde f}] (w'), [f \circ {\widetilde f}] (w')) \Longrightarrow [f \circ {\widetilde f}] (w') = 0\, .$$  

Thus,  $$(f \circ {\widetilde f}) (w) = \left \{ \begin{aligned} w \quad &\text{ if } \quad w\in \Im f \\ 0 \quad  &\text{ if } \quad w\in [\Im f]^\perp \end{aligned} \right .$$\noindent and, in particular, ${\widetilde f} (w') = 0$ when $w' \in [\text{Im } f]^\perp$.

In line with the above arguments, one has that $(f \circ {\widetilde f})^2 = f \circ {\widetilde f}$ and  $\text{Im }  ({\widetilde f} \circ f) = \text{Im } {\widetilde f}$.

Now, if $v\in \text{Im } {\widetilde f}$, $v = [{\widetilde f} \circ f] (\bar v)$ and $v' \in \Ker f$, then $$g(v,v') = g ( [{\widetilde f} \circ f] (\bar v), v') = g ({\bar v},0) = 0\, ,$$\noindent and we deduce that $\text{Im } {\widetilde f} \subseteq [\Ker f]^\perp$.

Finally, since $f_{\vert_{ [\Ker f]^\perp}}\colon [\Ker f]^\perp \overset {\sim} \longrightarrow \Im f$ and  $(f \circ {\widetilde f})_{\vert_{\text{Im } f}} = \text{Id}_{\vert_{\text{Im } f}}$, then $$({\widetilde f})_{\vert_{\text{Im } f}} = (f_{\vert_{ [\Ker f]^\perp}})^{-1}\Longrightarrow {\widetilde f} = f^\dagger\, .$$

Conversely, let us assume that there exists the Moore-Penrose inverse \linebreak $f^\dagger \colon W \to V$ of a linear map $f\colon V \to W$. Based on the same arguments as above one immediately has that:

\begin{itemize}

\item $f\circ f^\dagger$ and $f^\dagger \circ f$ are projections;

\item $\text{Im } (f\circ f^\dagger) = \text{Im } f$;

\item $[\text{Im } f]^\perp \subseteq \Ker (f\circ f^\dagger)$;

\item $\text{Im } (f^\dagger \circ f)  \subseteq [\Ker f]^\perp$.

\end{itemize}

Moreover, if $w\notin [\text{Im } f]^\perp$ there exists $\bar w \in W$ such that $$0 \ne  {\bar g} ([f \circ f^\dagger]^2 (\bar w), w) = {\bar g} ([f \circ f^\dagger] (\bar w), [f \circ f^\dagger] (w))\, ,$$\noindent from where we deduce that $w\notin \Ker (f\circ f^\dagger)$ and $[\text{Im } f]^\perp = \Ker (f\circ f^\dagger)$.

On the other hand it is clear that $\Ker f \subseteq \Ker (f^\dagger \circ f)$ and, if $v\in V$ with $f(v) \ne 0$ then $v \notin \Ker (f^\dagger \circ f)$, because $$f(v) = (f \circ f^\dagger \circ f) (v) \ne 0\, .$$

Hence, $\Ker f = \Ker (f^\dagger \circ f)$ and, bearing in mind that if $g\in \ed_k (V)$ is a projection then $V = \Ker g \oplus \text{Im } g$, one concludes that $$V = \Ker (f^\dagger \circ f) \oplus \text{Im } (f^\dagger \circ f) = \Ker f \oplus [\Ker f]^\perp$$\noindent and $$W =  \Ker (f \circ f^\dagger) \oplus \text{Im } (f \circ f^\dagger) =  \text{Im } f \oplus [\text{Im } f]^\perp\, .$$ Accordingly, $f$ is admissible for the Moore-Penrose inverse and the statement is deduced.

\end{proof}

\medskip

Since each isomorphism $g\colon V \to W$ is admissible for the Moore-Penrose inverse, a direct consequence of  Theorem \ref{th:MThTN} is that $g^\dagger = g^{-1}$, where $g^{-1}$ is the inverse map of $g$.

\smallskip

\begin{cor}  If  $(V,g)$ and $(W,{\bar g})$  are inner product spaces over $k$ and \linebreak  $f \colon V \to W$ is a linear map admissible for the Moore-Penrose inverse, then $f^\dagger$ is also admissible for the Moore-Penrose inverse and $(f^\dagger)^\dagger = f$.
\end{cor}

\begin{proof} This statement is deduced from Theorem \ref{th:MThTN} bearing in mind that:

\begin{itemize}

\item $\text{Im } f^\dagger = [\Ker f]^\perp$;

\item $[\text{Im } f^\dagger]^\perp = \Ker f$;

\item $\Ker f^\dagger = [\text{Im } f]^\perp$;

\item  $[\Ker f^\dagger]^\perp = \text{Im } f$.

\end{itemize}

\end{proof}

\smallskip

Moreover, if $f \colon V \to W$ is a linear map admissible for the Moore-Penrose inverse and $P_{[\ker f]^\perp}$ and $P_{\text{Im } f}$ are the projections induced by the decompositions $V = \Ker f \oplus [\Ker f]^\perp$ and $W = \text{Im } f \oplus [\text{Im } f]^\perp$, respectively, we obtain from the arguments of the proof of Theorem \ref{th:MThTN} that

\begin{cor}  If  $(V,g)$ and $(W,{\bar g})$  are inner product spaces over $k$ and \linebreak $f \colon V \to W$ is a linear map admissible for the Moore-Penrose inverse, then:

\begin{itemize}

\item $f^\dagger \circ f = P_{[\ker f]^\perp}$;

\item $f \circ f^\dagger = P_{\text{Im } f}$.

\end{itemize}

\end{cor}

\medskip

\subsection{Computation of the Moore-Penrose inverse of Endomorphisms on Arbitrary inner product Spaces}

\smallskip

Similar to the finite-dimensional situation, given an inner product space $(V, g)$ over $k$, $f \in \ed_{k} (V)$ let us assume that there exists a family of $f$-invariant finite-dimensional subspaces, ${\mathcal H}_f = \{H_i\}_{i\in I}$, such that $$V = \underset {i\in I} \bigoplus H_i\, .$$ 

Note that this assumption is always satisfied when $f$ admits an annihilator polynomial. 

To simplify, fixing a family  ${\mathcal H}_f$, we shall denote $f_i = f_{\vert_{H_i}}$.

\begin{defn} \label{def:reflexive-inf-decomp} We shall  call reflexive generalized inverse of $f$ associated with the family ${\mathcal H}_f$ to the unique linear map $f_{{\mathcal H}_f}^+ \in \ed_{k} (V)$ such that $[f_{{\mathcal H}_f}^+]_{\vert_{H_i}} = f_i^\dagger$ for every $i\in I$.
\end{defn}

For each vector $v\in V$, if $v = v_{i_1} + \dots + v_{i_s}$ with $v_{i_j} \in H_{i_j}$, then $$f_{{\mathcal H}_f}^+ (v) =  f_{i_1}^\dagger (v_1) + \dots +  f_{i_s}^\dagger (v_s)\, .$$

If $f$ is admissible for the Moore-Penrose inverse, our purpose is to determine when $f_{{\mathcal H}_f}^+ = f^\dagger$. To do this, the generalization onto infinite-dimensional vector spaces of Lemma \ref{lem:prop-dar-finite} is:

\begin{lem} \label{l:inf-cond} We have that:

\begin{enumerate}

\item $\Im f = \underset {i\in I} \bigoplus \Im f_i$;

\item If ${\mathcal H}_f^\perp = \underset {i\in I} \bigoplus [\Im f_i]_i^\perp$, then $V = \Im f \oplus {\mathcal H}_f^\perp$;

\item $\Ker f = \underset {i\in I} \bigoplus \Ker f_i$;

\item If ${\widetilde {\mathcal H}_f}^\perp = \underset {i\in I} \bigoplus [\Ker f_i]_i^\perp$, then $V = \Ker f \oplus {\widetilde {\mathcal H}_f}^\perp$;

\item $f$ induces an isomorphism between ${\widetilde {\mathcal H}_f}^\perp$ and $\Im f$.

\end{enumerate}

\end{lem}

\begin{proof} 

\quad

\begin{enumerate}

\item It is clear that  $\underset {i\in I} \bigoplus \Im f_i \subseteq \Im f$. Moreover, given $v \in \Im f$, if $v = f(v')$ with $v' = v'_{i_1} + \dots + v'_{i_s}$ and $v'_{i_j} \in H_{i_j}$ for all $j\in \{1, \dots, s\}$, then $$v = f(v'_{i_1}) + \dots + f(v'_{i_s}) \in  \underset {i\in I} \bigoplus \Im f_i \, ,$$\noindent because $f(v'_{i_j}) \in H_{i_j}$ for every $j\in \{1, \dots, s\}$.

\item  Since $H_i = \Im f_i \oplus [\Im f_i]_i^\perp$ for all $i\in I$, we have that $$\begin{aligned} V &= \underset {i\in I} \bigoplus H_i = \underset {i\in I} \bigoplus \big (\Im f_i \oplus [\Im f_i]_i^\perp \big ) = \\ &= \big (\underset {i\in I} \bigoplus \Im f_i \big ) \oplus \big (\underset {i\in I} \bigoplus [\Im f_i]_i^\perp \big ) = \Im f \oplus {\mathcal H}_f^\perp\, .\end{aligned}$$

\item It is immediate that  $\underset {i\in I} \bigoplus \Ker f_i \subseteq \Ker f$. Furthermore, if ${\bar v} \in \Ker f$ and $${\bar v} = {\bar v}_{i_1} + \dots + {\bar v}_{i_k}\quad \text{ with } \quad {\bar v}_{i_j} \in H_{i_j} \quad \text{ for every } j\in \{1, \dots, k\}\, ,$$ since $$ f({\bar v}_{i_1}) + \dots + f({\bar v}_{i_k}) = 0 \text{ and } H_{i_j} \cap [\sum_{r\ne j} H_{i_r}] = {0}\, ,$$\noindent we conclude that $f({\bar v}_{i_j}) = 0$ for all $j\in \{1, \dots, k\}$, from which it is deduced that $\Ker f \subseteq \underset {i\in I} \bigoplus \Ker f_i$.

\item Similar to (2).

\item It is a direct consequence of (2) and (4).

\end{enumerate}

\end{proof}

\medskip

It is now easy to prove that the generalization onto an arbitrary vector space $V$ of the Lemma \ref{l:or-dcomp} is the following:

\begin{lem} \label{l:or-dcomp-2} With the previous assumptions on $V$ and ${\mathcal H}_f = \{U_i\}_{i\in I}$, if $U\subset V$ is a subspace and $\{U_i\}_{i\in I}$ is a family of subspaces of $U$ such that $U = \underset {i\in I} \oplus U_i$ with $U_i \subseteq H_i$, then $$\underset {i\in I} \oplus [U_i]_i^\perp \subseteq U^\perp \Longleftrightarrow 
[U_i]_i^\perp \subset [\sum_{j\ne i} U_j]^\perp \text{ for all } i\in I\, .$$

\end{lem}

\medskip

Accordingly  we have that:

\begin{lem} \label{lem:equ-imf-inf} If $(V, g)$ is an inner product vector space over $k$, $f \in \ed_{k} (V)$, and ${\mathcal H}_f = \{H_i\}_{i\in I}$ with $V = \underset {i\in I} \bigoplus H_i$ and each $H_i$ is $f$-invariant, then ${\mathcal H}_f^\perp = [\Im f]^\perp$ if and only if $$[\Im f_i]_i^\perp \subseteq [\sum_{j\ne i} \Im f_j]^\perp \text{ for every }i \in I\, .$$

\end{lem}

\begin{proof} This statement is the generalization of Lemma \ref{lem:equ-imf-345} to arbitrary vector spaces.

\end{proof}

Moreover, similar to Lemma \ref{lem:equ-kerf-345} one has that: 

\begin{lem} \label{lem:equ-kerf-inf}  If $(V, g)$ is an inner product vector space over $k$, $f \in \ed_{k} (V)$, and ${\mathcal H}_f = \{H_i\}_{i\in I}$ with $V =   \underset {i\in I} \bigoplus H_i$ and each subspace $H_i$ is $f$-invariant, then ${\widetilde {\mathcal H}_f}^\perp = [\Ker f]^\perp$ if and only if  $$[\Ker f_i]_i^\perp \subseteq [\sum_{j\ne i} \Ker f_j]^\perp \text{ for every }i \in I\, .$$

\end{lem}

\medskip

\begin{thm} \label{t:char}  If $(V, g)$ is  an inner product vector space over $k$, $f \in \ed_{k} (V)$ is admissible  for the Moore-Penrose inverse, and ${\mathcal H}_f = \{H_i\}_{i\in I}$ with $V = \underset {i\in I} \bigoplus H_i$ and each subspace $H_i$ is $f$-invariant, then  $f_{{\mathcal H}_f}^+ = f^\dagger$ if and only if the following conditions are satisfied:

\begin{enumerate}

\item $[\Im f_i]_i^\perp \subseteq [\sum_{j\ne i} \Im f_j]^\perp$ for every $i \in I$;

\item $[\Ker f_i]_i^\perp \subseteq [\sum_{j\ne i} \Ker f_j]^\perp$ for all $i \in I$.

\end{enumerate}

\end{thm}

\begin{proof} The claim is a direct consequence of Lemma \ref{lem:equ-imf-inf} and Lemma \ref{lem:equ-kerf-inf}.
\end{proof}

\begin{cor} \label{c:char}   If $(V, g)$ is an inner product vector space over $k$, $f \in \ed_{k} (V)$, ${\mathcal H}_f = \{H_i\}_{i\in I}$ with $V = \underset {i\in I} \bigoplus H_i$ and each subspace $H_i$ is $f$-invariant, and $H_i \subseteq  [\sum_{j\ne i} H_j]^\perp$ for every $i\in I$, then $f$ is admissible  for the Moore-Penrose inverse and $f_{{\mathcal H}_f}^+ = f^\dagger$.
\end{cor}

\begin{proof} With the hypothesis of this Corollary the conditions of Theorem  \ref{t:char} are satisfied. 

\end{proof}

\medskip

\begin{exam} \label{ex:moore-fini-pot}

Let $(V,g)$ be an inner product vector space of countable dimension over $k$. Let $\{{ {v_1}},{ {v_2}},{ {v_3}},\dots\}$ be an orthonormal basis of $V$ indexed by the natural numbers. 

Let $\varphi \in \ed_{k} (V)$ the finite potent endomorphism defined as follows: 
$$\varphi ({ {v_i}}) = \left \{\begin{aligned} { {v_{2}}} +  { {v_{5}}}   +  { {v_{7}}} \quad   &\text{ if } \quad i = 1 \\   { {v_{1}}} +  3{ {v_{2}}} \quad   \quad   &\text{ if } \quad i = 2 \\    { {v_{4}}}  \qquad   \quad   &\text{ if } \quad i = 3 \\   { {v_{1}}} - { {v_{3}}} \quad   \quad   &\text{ if } \quad i = 4 \\ -{ {v_{3}}} +  2{ {v_{5}}}   +  2{ {v_{7}}} \quad   &\text{ if } \quad i = 5 \\    3{ {v_{i+1}}} \qquad \quad  &\text{ if } \quad i = 5h + 1 \\  { {0}} \qquad \quad   &\text{ if } \quad i = 5h + 2 \\ -{ {v_{i-2}}} + 2{ {v_{i+1}}}  \quad  &\text{ if } \quad i = 5h + 3 \\ { {v_{i-2}}} + { {v_{i+1}}} \quad   &\text{ if } \quad i = 5h + 4 \\ - { {v_{i-4}}} + 5{ {v_{i-3}}}  \quad  &\text{ if } \quad i = 5h + 5    \end{aligned} \right .$$\noindent for all $h\geq 1$.

We can consider the decomposition $V = \bigoplus_{i\in \mathbb N} H_i$ where $$H_1 = \langle v_i\rangle_{i\in \{1,\dots,10\}} \text{ and }  H_j = \langle v_{5j+1}, v_{5j+2}, v_{5j+3}, v_{5j+4}, v_{5j+5}  \rangle$$\noindent for every $j\geq 2$. It is clear that $H_i$ is a $\varphi$-invariant subspace of $V$ for every $i\in \mathbb N$.

In the above bases one has that $$\varphi_{\vert_{H_1}}\equiv \begin{pmatrix} 0 & 1 & 0 & 1 & 0 & 0 & 0 & 0 & 0 & 0 \\  1 & 3 & 0 & 0 & 0 & 0 & 0 & 0 & 0 & 0 \\ 0 & 0 & 0 & -1 & -1 & 0 & 0 & 0 & 0 & 0 \\ 0 & 0 & 1 & 0 & 0 & 0 & 0 & 0 & 0 & 0 \\ 1 & 0 & 0 & 0 & 2 & 0 & 0 & 0 & 0 & 0 \\  0 & 0 & 0 & 0 & 0 & 0 & 0 & -1 & 0 & -1 \\
 1 & 0 & 0 & 0 & 2 & 3 & 0 & 0 & 1 & 5 \\
 0 & 0 & 0 & 0 & 0 & 0 & 0 & 0 & 0 & 0 \\
 0 & 0 & 0 & 0 & 0 & 0 & 0 & 2 & 0 & 0 \\
 0 & 0 & 0 & 0 & 0 & 0 & 0 & 0 & 1 & 0  \end{pmatrix}$$ \noindent 

and $$\varphi_{\vert_{H_i}}\equiv \begin{pmatrix}
0 & 0 & -1 & 0 & -1 \\
3 & 0 & 0 & 1 & 5 \\
0 & 0 & 0 & 0 & 0 \\
0 & 0 & 2 & 0 & 0 \\
0 & 0 & 0 & 1 & 0 
\end{pmatrix}$$\noindent  for all $i\geq 2$.

 Thus, bearing in mind that:

\begin{itemize}

\item $\Im \varphi = \langle v_1, v_2, v_3, v_4, v_5 \rangle \oplus [ \underset {i\geq 1} \oplus \langle v_{5j+1}, v_{5j+2},  v_{5j+4}, v_{5j+5}  \rangle]$ and $[\Im \varphi]^\perp =\underset {i\geq 1} \oplus \langle {v_{5i+3}}  \rangle$;

\item $\Ker \varphi= \underset {i\geq 1} \oplus  \langle v_{5i +2} \rangle$ and $[\Ker \varphi]^\perp = \langle v_1, v_2, v_3, v_4, v_5 \rangle \oplus [\underset {i\geq 1} \oplus \langle v_{5j+1}, v_{5j+3},  v_{5j+4}, v_{5j+5}\rangle]$;

\end{itemize}

we have that $\varphi$ is admissible for the Moore-Penrose inverse and, since $\{v_i\}_{i\in I}$ is an orthonormal basis of $V$, Corollary \ref{c:char}  holds for $\varphi$.

Accordingly, a non-difficult computation shows that $$(\varphi^\dagger)_{\vert_{H_1}} \equiv \begin{pmatrix}
6 & -2 & 6 & 0 & 3 & 0 & 0 & 0 & 0 & 0 \\
-2 & 1 & -2 & 0 & -1 & 0 & 0 & 0 & 0 & 0 \\
0 & 0 & 0 & 1 & 0 & 0 & 0 & 0 & 0 & 0 \\
3 & -1 & 2 & 0 & 1 & 0 & 0 & 0 & 0 & 0 \\
-3 & 1 & -3 & 0 & -1 & 0 & 0 & 0 & 0 & 0 \\[4pt]  0 & 0 & 0 & 0 & -\frac13 & \frac{5}{3} & \frac{1}{3} & 0 & \frac{5}{6} & -\frac{1}{3} \\[4pt]
 3 & -1 & 2 & 0 & -1 & 0 & 0 & 0 & 0 & 0 \\[4pt]
 0 & 0 & 0 & 0 & 0 & 0 & 0 & 0 & \frac{1}{2} & 0 \\[4pt]
 0 & 0 & 0 & 0 & 0 & 0 & 0 & 0 & 0 & 1 \\[4pt]
 0 & 0 & 0 & 0 & 0 & -1 & 0 & 0 & -\frac{1}{2} & 0 
\end{pmatrix}$$ 

and

 $$ (\varphi^\dagger)_{\vert_{H_i}}\equiv \begin{pmatrix}
\frac{5}{3} & \frac{1}{3} & 0 & \frac{5}{6} & -\frac{1}{3} \\
0 & 0 & 0 & 0 & 0 \\
0 & 0 & 0 & \frac{1}{2} & 0 \\
0 & 0 & 0 & 0 & 1 \\
-1 & 0 & 0 & -\frac{1}{2} & 0 
\end{pmatrix}$$\noindent for all $i\geq 2$, from where the endomorphism $\varphi^\dagger$ is determined.

Thus, the explicit expression of $\varphi^\dagger$ is

$$\varphi^\dagger ({{v_{i}}}) = \left \{\begin{aligned}6{ {v_{1}}} - 2 { {v_{2}}} + 3{ {v_{4}}} -3  { {v_{5}}}   - 3 { {v_{7}}} \qquad \quad   &\text{ if } \quad i = 1 \\   -2{ {v_{1}}} + { {v_{2}}} - { {v_{4}}} +  { {v_{5}}}   +  { {v_{7}}} \qquad   \quad   &\text{ if } \quad i = 2 \\    6{ {v_{1}}} - 2 { {v_{2}}} + 2{ {v_{4}}} -3  { {v_{5}}}   - 3 { {v_{7}}} \qquad   \quad   &\text{ if } \quad i = 3 \\    { {v_{3}}}  \qquad \qquad \qquad \qquad   \quad   &\text{ if } \quad i = 4 \\ 3{ {v_{1}}} -  { {v_{2}}} + { {v_{4}}} -  { {v_{5}}} -\frac13 v_6   -  { {v_{7}}} \qquad \quad   &\text{ if } \quad i = 5 \\  \frac53 { {v_{i}}}  -{{v_{i+4}}}  \qquad  \qquad  \quad  &\text{ if } \quad i = 5h + 1 \\[4pt]  \frac13 { {v_{i-1}}} \qquad  \qquad  \qquad \quad   &\text{ if } \quad i = 5h + 2 \\[4pt] \qquad  { {0}} \qquad   \qquad \qquad \qquad  \quad  &\text{ if } \quad i = 5h + 3 \\[4pt] \frac56 { {v_{i-3}}}+\frac{1}{2} v_{i-1}-\frac12 { {v_{i+1}}}\ \quad\quad   &\text{ if } \quad i = 5h + 4 \\[4pt] - \frac13 { {v_{i-4}}} +{ {v_{i-1}}} \quad  \qquad  \quad  &\text{ if } \quad i = 5h + 5    \end{aligned} \right . $$\noindent for all $h\geq 1$.
\end{exam}

\medskip

\begin{rem} We wish point out that this example shows that the Moore-Penrose inverse of a finite potent endomorphism is not, in general, a finite potent endomorphism.
\end{rem}

\medskip

\section{Study of Infinite systems of linear equations from the Moore-Penrose Inverse} \label{s:study}

The aim of this final section is to study solutions of infinite system of linear equations from the Moore-Penrose inverse of linear maps characterized in the previous section.

\begin{defn} If $V$ and $W$ are two arbitrary $k$-vector spaces and $f\colon V \to W$ is a linear map, a  linear system  is an expression $$f (x) = w\, ,$$\noindent where $w\in W$. This system is called ``consistent'' when $w\in \Im f$.
\end{defn}

If  $V$ and $W$ are infinite-dimensional vector spaces, fixing bases of $V$ and $W$, the linear system $f (x) = w$ is equivalent to an infinite system of linear equations.

If a linear system $f (x) = w$ is consistent and $f (v_0) = w$ for a certain $v_0\in V$, then the set of solutions of this system is $v_0 + \Ker f$. The vector $v_0$ is named ``particular solution'' of the system.

Let us now consider two arbitrary  inner product vector spaces $(V,g)$ and $(W,{\bar g})$  over $k$, let $f \colon V \to W$ be a linear map admissible for the Moore-Penrose inverse and let $f^\dagger$ be its Moore-Penrose inverse.  If $f (x) = w$ is a linear system , since $f^\dagger$ is a reflexive generalized inverse of $f$, then $f$ is consistent if and only if $(f \circ f^\dagger) (w) = w$ and, in this case, the set of solutions of the linear system is $$f^\dagger (w) + \Ker f\, .$$

In finite-dimensional inner product vector spaces it is known that the Moore-Penrose inverse is useful for studying the least squares solutions of a linear system. To complete this work, we shall generalize this notion to arbitrary vector spaces.

\begin{defn} If $(V,g)$ and $(W,{\bar g})$  are two arbitrary inner product vector spaces over $k$, then ${v'} \in V$ is called ``least ${\bar g}$-norm solution'' of a linear system $f (x) = w$ when $$\Vert f(v') - w\Vert_{\bar g} \leq \Vert f({v}) - w\Vert_{\bar g}$$\noindent for all $v\in V$.
\end{defn}

Note that ${ v'} \in V$ is a least ${\bar g}$-norm solution of the linear system $f (x) = w$ if and only if $d_{\bar g} (w, f({v'})) = d_g (w, \Im f)$.

\begin{defn} If $(V,g)$ and $(W,{\bar g})$  are two arbitrary inner product vector spaces over $k$, then ${\tilde v} \in V$ is called ``minimal least ${\bar g}$-norm solution'' of a linear system $f (x) = w$ when $$\Vert {\tilde v} \Vert_{g} \leq \Vert v'\Vert_{g}$$\noindent for every least ${\bar g}$-norm solution $v' \in V$.
\end{defn}

\medskip

\begin{prop} \label{p:study-least} If $(V,g)$ and $(W,{\bar g})$ are two arbitrary inner product vector spaces over $k$, $f \colon V \to W$ is a linear map admissible for the Moore-Penrose inverse and $f (x) = w$ is a linear system, then $f^\dagger (w)$ is the unique minimal least ${\bar g}$-norm solution of this linear system.
\end{prop}

\begin{proof} Firstly, since $\Im (f \circ f^\dagger - \text{Id}) \subseteq [\Im f]^\perp$, one has that \begin{equation} \label{eq:least} \begin{aligned} \Vert f({v}) - w\Vert_{\bar g}^2 &= \Vert [f({v}) - f(f^\dagger (w))] +  [f(f^\dagger (w)) - w]\Vert_{\bar g}^2= \\ &= \Vert f({v}) - f(f^\dagger (w))\Vert_{\bar g}^2 + \Vert  f(f^\dagger (w)) - w\Vert_{\bar g}^2 \end{aligned} \end{equation} for all $v\in V$. Hence, $$\Vert f(f^\dagger(w)) - w\Vert_{\bar g} \leq \Vert f({v}) - w\Vert_{\bar g}$$\noindent for all $v\in V$, and we deduce that  $f^\dagger (w)$ is a least ${\bar g}$-norm solution of $f (x) = w$.

Moreover, it follows from (\ref{eq:least}) that $v'\in V$ is a least ${\bar g}$-norm solution of this linear system if and only if $f(v') - f (f^\dagger (w)) = 0$, that is, $v'$ is a solution of the consistent system $$f(x) - f(f^\dagger (w)) = 0\, .$$ Thus, for each  least $\bar g$-norm solution $v' \in V$, one has that $$f^\dagger (w) = v' + h\, ,$$\noindent with $h\in \Ker f$ and, bearing in mind that $f^\dagger (w) \in [\Ker f]^\perp$, we conclude that $f^\dagger (w)$ is the unique minimal least $\bar g$-norm solution of $f(x) = w$ because $$\Vert f^\dagger (w) \Vert_g  < \Vert v'\Vert_g\, ,$$\noindent for every  $v' \ne f^\dagger (w)$.
\end{proof}

\medskip

\begin{exam}
Let $(V,g)$ be an inner product vector space of countable dimension over $k$. Let $\{{ {v_1}},{ {v_2}},{ {v_3}},\dots\}$ be an orthonormal basis of $V$ indexed by the natural numbers. If $(x_i)_{i\in\mathbb{N}} \in \underset {i\in \mathbb N} \bigoplus k$, since $x_i=0$ for almost all $i\in\mathbb{N}$, we shall write $x = (x_i)$ to denote the well-defined vector $$x = \sum_{i\in \mathbb N} x_i\cdot v_i \in V\, .$$

Let $\varphi \in \ed_{k} (V)$ the finite potent endomorphism studied in Example \ref{ex:moore-fini-pot}. We can consider the system $$\varphi(x)=w,$$ where $w=(\alpha_i)_{i\in\mathbb{N}}$ and whose explicit expression is: 
\begin{equation} \label{eq:final} \left \{ \begin{aligned} x_2 + x_4  &= \alpha_1 \\ x_1 + 3x_2 &=  \alpha_2   \\ -x_4 - x_5 &=  \alpha_3  \\  x_3 &=  \alpha_4  \\  x_1 + 2x_5 &=  \alpha_5 \\  -x_8 - x_{10} &= \alpha_6 \\  x_1 + 2x_5 + 3x_6 + x_9 + 5 x_{10} &= \alpha_7  \\   0 &= \alpha_{5h+3}  \text{ for all } h \geq 1 \\ 2x_{5h +3} &= \alpha_{5h+4}    \text{ for all } h \geq 1 \\  x_{5h + 4} &= \alpha_{5h+5}   \text{ for all } h \geq 1 \\ -x_{5h+3} - x_{5h+5} &=  \alpha_{5h+1}   \text{ for all } h \geq 2 \\  3x_{5h + 1} + x_{5h+4} + 5x_{5h+5} &=  \alpha_{5h+2} \text{ for all } h \geq 2 \end{aligned} \right . \, . \end{equation}\\

 By Proposition \ref{p:study-least}, we have that the unique minimal least $g$-norm solution of this linear system is $(\beta_i)_{i\in\mathbb{N}} = \varphi^\dagger (w)$ and, bearing in mind the expression of $\varphi^\dagger$ obtained in  Example \ref{ex:moore-fini-pot}, an easy computation shows that $$\beta_i=\left \{\begin{aligned} 
6\alpha_1-2\alpha_2+6\alpha_3+3\alpha_5 \quad \quad    &\text{ if } \quad i = 1 \\   
-2\alpha_1+\alpha_2-2\alpha_3-\alpha_5 \quad  \quad \quad   &\text{ if } \quad i = 2 \\    
\alpha_4  \qquad\qquad\qquad\quad   &\text{ if } \quad   i = 3 \\   
3\alpha_1-\alpha_2+2\alpha_3+\alpha_5 \qquad   \quad   &\text{ if } \quad i = 4 \\ 
-3\alpha_1+\alpha_2-3\alpha_3-\alpha_5 \qquad\quad   &\text{ if } \quad i = 5 \\    
-\frac{1}{3}\alpha_5+\frac{5}{3}\alpha_6+\frac{1}{3}\alpha_7+\frac{5}{6}\alpha_9-\frac{1}{3}\alpha_{10} \quad  &\text{ if } \quad i = 6 \\  
3\alpha_1 - \alpha_2 + 2\alpha_3 - \alpha_5 \qquad\quad  &\text{ if } \quad i = 7 \text{, for all h} \geq 1 \\ 
\frac{1}{2}\alpha_{i+1} \qquad\qquad\qquad  &\text{ if } \quad i = 5h + 3 \text{, for all h} \geq 1 \\ 
\alpha_{i+1} \qquad\qquad\qquad   &\text{ if } \quad i = 5h + 4 \text{, for all h} \geq 1 \\ 
-\alpha_{i-4}-\frac{1}{2}\alpha_{i-1}  \qquad\qquad  &\text{ if } \quad i = 5h + 5 \text{, for all h} \geq 1 \\
\frac{5}{3}\alpha_i+\frac{1}{3}\alpha_{i+1}+\frac{5}{6}\alpha_{i+3}-\frac{1}{3}\alpha_{i+4} \quad  &\text{ if } \quad i = 5h + 1 \text{, for all h} \geq 2 \\  
0 \qquad\qquad\qquad\quad  &\text{ if } \quad i = 5h + 2 \text{, for all h} \geq 2   \end{aligned} \right .\, .$$

Moreover, one have that
$$(\varphi\circ\varphi^\dagger) ({v_i}) = \left \{\begin{aligned}{0}   \quad   &\text{ if } \quad i = 5h+3 \\
 { v_{i}} \quad   &\text{ otherwise }   \end{aligned} \right .\, .$$
 
Thus, according to (\ref{eq:char-im-rel}), the system (\ref{eq:final}) is consistent if and only if $\alpha_{5h+3}=0$ for all $h\geq 1$, and, in this case, $(\beta_i)_{i\in\mathbb{N}}$ is a particular solution of it.
\end{exam}

\end{document}